\documentclass{article}
\usepackage{graphicx}
\usepackage[none]{hyphenat}
\usepackage{amsmath}
\usepackage{amssymb}
\usepackage{amsthm}
\usepackage{mathtools}

\title{A surprising generalization of the M\"{o}bius function}
\author{George E. Andrews, Louis H. Kauffman, Divyamaan Sahoo}
\date{}

\newtheorem{thm}{Theorem}
\newtheorem{lem}{Lemma}
\newtheorem{prop}{Proposition}

\begin{document}

\maketitle

\begin{center}
    \textit{In dedication to J. M. Flagg (1949--2025)}
\end{center}

\setlength{\parindent}{0pt}
\setlength{\parskip}{1em}

\begin{abstract}
\noindent We introduce a broad generalization of a recursive formula for the M\"{o}bius function $\mu(n) = 1-n-\sum_{d=2}^{n-1}\mu(d)\left[\frac{n}{d}\right]$ due to George Spencer-Brown. By replacing the greatest integer function $\left[\frac{n}{d}\right]$ in this classical recurrence with an arbitrary arithmetic function $f\left(\left[\frac{n}{d}\right]\right)$ with $f(1)=1$, we define a new generalized family of functions, denoted $\star(n)$. We prove that $\star(p) = -1$ if and only if $p$ is prime. This result yields a surprising algebraic characterization of primes and reveals a deep structural property underlying divisor sums and the greatest integer function.
\end{abstract}


\section{Introduction}

The M\"{o}bius function $\mu$ is at the heart of number theory and physics \cite{chen2010}. Beyond its classical role in multiplicative number theory, M\"{o}bius inversion appears naturally in inverse problems, statistical physics, and lattice systems \cite{chen2010}. 

\noindent In 1826, E. Meissel \cite{meissel1826} observed:
$$\sum_{d=1}^n \mu(d)\left[\frac{n}{d}\right]=1,$$
where $\left[{.}\right]$ denotes the greatest integer function. In 2012, G. Spencer-Brown \cite{flagg} reformulated this unique identity to reveal a recursive formula for $\mu(n)$ that does not require knowledge of the prime factorization of $n$:
$$\mu(n) = 1-n-\sum_{d=2}^{n-1} \mu(d)\left[\frac{n}{d}\right]$$

Computation of the M\"{o}bius function using this recursive formula is significantly more efficient than traditional sieve methods, as demonstrated by D. Sahoo \cite{sahoo}. Sahoo constructed a family of recursive functions $\star_k$, for any positive integer $k$: 
$$\star_k (n)=1-n-\sum_{d=2}^{n-1}\star_k(d)\left[\frac{n}{d}\right]^k,$$
and conjectured that $\star_k(p) = -1$ iff $p$ is prime. Through numerical experiments, he observed that this conjecture holds for $k\in \mathbb{Q}, \mathbb{R}$, and $\mathbb{C}$, revealing an uncountably infinite family of primality tests.
 
This led to a much more general conjecture. Let $\star:\mathbb{N}\to\mathbb{N}$ and $f:\mathbb{N}\to\mathbb{N}$, with $\star(1)=1$ and $f(1)=1$. Then define:
\begin{equation}
    \boxed{\star(n)=1-n-\sum_{d=2}^{n-1}\star(d) f\left(\left[\frac{n}{d}\right]\right)}
    \label{eq:star}
\end{equation}

We note that

$\star(2)$\\ 
$= 1 - 2 - 0$\\
$= -1$.

$\star(3)$\\ 
$= 1 - 3 - \star(2).f(1)$\\
$= 1 - 3 - (-1).1$\\
$= -1.$

$\star(4)$\\
$= 1 - 4 - \star(2).f(2) - \star(3).f(1)$\\
$= 1 - 4 - (-1).f(2) - (-1).1$\\
$= -2 + f(2)$.

$\star(5)$\\
$= 1 - 5 - \star(2).f(2) - \star(3).f(1) - \star(4).f(1)$\\
$= 1 - 5 - (-1).f(2) - (-1).1 - (-2 + f(2)).1$\\
$= -1$.

It would appear that $\star(p)=-1$ for all primes $p$, and the object of this paper is to prove this remarkable fact.\\

\begin{thm}
$\mathbf{\star(p)=-1}$ \textbf{iff} $\mathbf{p}$ \textbf{is prime.}
\end{thm}

In addition to Theorem 1, we observe:\\
$\star(p^2) = f(p) - f(p-1) - 1$ iff $p$ is a prime, and\\
$\star(pq)=f(p)-f(p-1)+f(q)-f(q-1)-1$ iff $p$ and $q$ are distinct primes.

This property of $\star(pq)$ follows from the reformulation of $\star$ as a sum over divisors (Proposition 1) introduced in the following section.

\section{Proof of Theorem 1}

This proof is divided into two parts. First, we prove Proposition 1, that $\star$ can be reformulated as a sum over divisors. Then, we use Proposition 1 to prove Theorem 1, that $\star(p)=-1$ iff $p$ is prime.

We begin this section with some background. Assume $\star:\mathbb{N}\to\mathbb{N}$ and $f:\mathbb{N}\to\mathbb{N}$, with $\star(1)=1$ and $f(1)=1$. By definition \eqref{eq:star},
\begin{equation*}
    \boxed{\star(n)=1-n-\sum_{d=2}^{n-1}\star(d) f\left(\left[\frac{n}{d}\right]\right)}
\end{equation*}

First consider, 
$$\star(n)=1-n-\sum_{d=2}^{n-1}\star(d) f\left(\left[\frac{n}{d}\right]\right)$$
$$\implies \sum_{d=2}^{n-1}\star(d) f\left(\left[\frac{n}{d}\right]\right)+\star(n)=1-n$$
$$\implies f(n)+\sum_{d=2}^{n-1}\star(d) f\left(\left[\frac{n}{d}\right]\right)+\star(n)=f(n)+1-n$$
\begin{equation}
    \label{eq:n}
   \implies \sum_{d=1}^{n}\star(d) f\left(\left[\frac{n}{d}\right]\right)=f(n)+1-n 
\end{equation}

Now, replacing $n$ by $n-1$ in Equation~\eqref{eq:n} gives:
\begin{equation}
    \label{eq:n-1}
    \sum_{d=1}^{n-1}\star(d) f\left(\left[\frac{n-1}{d}\right]\right)=f(n-1)+1-(n-1)
\end{equation}

Subtracting Equation~\eqref{eq:n-1} from Equation~\eqref{eq:n}, we have: 
$$\sum_{d=1}^{n}\star(d) f\left(\left[\frac{n}{d}\right]\right) - \sum_{d=1}^{n-1}\star(d) f\left(\left[\frac{n-1}{d}\right]\right) = f(n)+1-n - \left(f(n-1)+1-(n-1)\right)$$
$$\implies \star(n)+\sum_{d=1}^{n-1}\star(d) f\left(\left[\frac{n}{d}\right]\right) - \sum_{d=1}^{n-1}\star(d) f\left(\left[\frac{n-1}{d}\right]\right) = f(n)-f(n-1)-1$$
$$\implies \star(n)+\sum_{d=1}^{n-1}\star(d)\left( f\left(\left[\frac{n}{d}\right]\right)-f\left(\left[\frac{n-1}{d}\right]\right)\right) = f(n)-f(n-1)-1$$
$$\implies \star(n)+\star(1)\left( f\left(\left[\frac{n}{1}\right]\right)-f\left(\left[\frac{n-1}{d}\right]\right)\right)+\sum_{d=2}^{n-1}\star(d)\left( f\left(\left[\frac{n}{d}\right]\right)-f\left(\left[\frac{n-1}{d}\right]\right)\right)$$
$$= f(n)-f(n-1)-1$$
$$\implies \star(n)+f(n)-f(n-1)+\sum_{d=2}^{n-1}\star(d)\left( f\left(\left[\frac{n}{d}\right]\right)-f\left(\left[\frac{n-1}{d}\right]\right)\right)= f(n)-f(n-1)-1$$
$$\implies \star(n)+\sum_{d=2}^{n-1}\star(d)\left( f\left(\left[\frac{n}{d}\right]\right)-f\left(\left[\frac{n-1}{d}\right]\right)\right)= -1$$
This implies the following equivalent reformulation of the star function:
\begin{equation}
    \boxed{\star(n)= -1-\sum_{d=2}^{n-1}\star(d)\left( f\left(\left[\frac{n}{d}\right]\right)-f\left(\left[\frac{n-1}{d}\right]\right)\right)}
\end{equation}

Now consider two lemmas:

\begin{lem}
    $\left[\frac{n}{d}\right]-\left[\frac{n-1}{d}\right] = 1$ iff $d|n$, else $=0$.
\end{lem}

\begin{proof}
The rational numbers $\frac{n}{d}$ and $  \frac{n-1}{d}$ differ by $\frac{1}{d}$. The interval $\left(\frac{n-1}{d}, \frac{n}{d}\right]$ has length $\frac{1}{d} \le 1$, since $d$ is a positive integer. This interval contains an integer iff $  \frac{n}{d}$ is an integer. 

If $d|n$, then $  \frac{n}{d}$ is an integer, say $k$. It follows that $  \frac{n-1}{d} = k - \frac{1}{d}$. Since $d \ge 1$, it follows that $\left[ k - \frac{1}{d}\right]=k-1$. Thus, $  \left[\frac{n}{d}\right]-\left[\frac{n-1}{d}\right] = k - (k-1) = 1$.

If $d \nmid n$, then $  \frac{n}{d}$ is not an integer. No integer lies within the interval $  \left(\frac{n-1}{d}, \frac{n}{d}\right]$. This implies that $  \frac{n}{d}$ and $  \frac{n-1}{d}$ must lie between the same two consecutive integers, so $  \left[\frac{n-1}{d}\right] = \left[\frac{n}{d}\right]$, and $  \left[\frac{n}{d}\right] - \left[\frac{n-1}{d}\right]=0$.
\end{proof}

\begin{lem}
  $  f\left(\left[\frac{n}{d}\right]\right) - f\left(\left[\frac{n-1}{d}\right]\right) = 0$, if $d\nmid n$.  
\end{lem}
\begin{proof}
If $d\nmid n$ then
$  \left\lfloor\frac{n}{d}\right\rfloor = \left\lfloor\frac{n-1}{d}\right\rfloor$ by Lemma 1. So $  f\left(\left[\frac{n}{d}\right]\right) - f\left(\left[\frac{n-1}{d}\right]\right)= 0.$
\end{proof}

Using Lemmas 1 and 2, the star function is reformulated as a sum over divisors:
$$\star(n)= -1-\sum_{d=2}^{n-1}\star(d)\left( f\left(\left[\frac{n}{d}\right]\right)-f\left(\left[\frac{n-1}{d}\right]\right)\right)$$
$$\implies \star(n)= -1-\sum_{d=2, d|n}^{n-1}\star(d)\left( f\left(\left[\frac{n}{d}\right]\right)-f\left(\left[\frac{n-1}{d}\right]\right)\right)$$
$$-\sum_{d=2, d\nmid n}^{n-1}\star(d)\left( f\left(\left[\frac{n}{d}\right]\right)-f\left(\left[\frac{n-1}{d}\right]\right)\right)$$

By Lemma 2, $\sum_{d=2, d\nmid n}^{n-1}\star(d)\left( f\left(\left[\frac{n}{d}\right]\right)-f\left(\left[\frac{n-1}{d}\right]\right)\right) = 0$. Hence, the star function can be reformulated as a sum over divisors. This is used to prove Theorem 1.
\begin{prop}
    \begin{equation}
    \boxed{\star(n)= -1-\sum_{d=2, d|n}^{n-1}\star(d)\left( f\left(\left[\frac{n}{d}\right]\right)-f\left(\left[\frac{n-1}{d}\right]\right)\right)}
    \end{equation}
\end{prop}

\begin{proof}[Proof of Proposition 1]
\hphantom{.}\\
The proof is given in the discussion
preceding the statement of the Proposition.
\end{proof}

\begin{proof}[\large Proof of Theorem 1]
\hphantom{.}\\
$(\impliedby)$\\
Let $p$ be a prime number. It follows that
$$\star(p)= -1-\sum_{d=2, d|p}^{p-1}\star(d)\left( f\left(\left[\frac{p}{d}\right]\right)-f\left(\left[\frac{p-1}{d}\right]\right)\right).$$
For any prime $p$, there is no $d$ where $d|p$ and $2\leq d \leq p-1$.\\
Hence, $\sum_{d=2, d|p}^{p-1}\star(d)\left( f\left(\left[\frac{p}{d}\right]\right)-f\left(\left[\frac{p-1}{d}\right]\right)\right) = 0$. It follows that 
$$\star(p)= -1.$$
$(\implies)$\\
Consider the contrapositive: if $n$ is composite, then $\star(n)$ is not identically $-1$ for all valid functions $f$.

Assume $n$ is a composite number. Then $n$ has at least one proper divisor. Let $p$ be the smallest proper divisor of $n$. By definition, $p$ must be prime.

Let $q = \frac{n}{p}$. Because $p \ge 2$ is strictly the smallest proper divisor of $n$, it follows that $q$ is strictly the largest proper divisor of $n$.

Now reconsider the reformulated star function (Proposition 1). Since the summation is over $d$ where $d|n$, we can further simplify $\left[\frac{n}{d}\right] = \frac{n}{d}$ and $\left[\frac{n-1}{d}\right] = \frac{n}{d}-1$. Thus, the reformulation can further simplify to:
\begin{equation}
    \boxed{\star(n)= -1-\sum_{d=2, d|n}^{n-1}\star(d)\left( f\left(\frac{n}{d}\right)-f\left(\frac{n}{d}-1\right)\right)}
\end{equation}

Now consider the summation when $d=p$. Since $p$ is prime, we know from the $(\impliedby)$ direction of this proof that $\star(p) = -1$. Hence,
$$-\star(p)\left( f\left(\frac{n}{p}\right)-f\left(\frac{n}{p}-1\right)\right) = -(-1)\left( f(q)-f(q-1)\right) = f(q)-f(q-1),$$
where $q=\frac{n}{d}=\left[\frac{n}{d}\right]$. This introduces the term $+f(q)$ into the polynomial expansion of $\star(n)$. 

For this $f(q)$ term to be canceled out of the final polynomial, there must exist some other proper divisor $k$ in the summation (where $k \neq p$) that produces a $-f(q)$ term. The terms generated by any divisor $k$ are of the form $-\star(k)f\left(\frac{n}{k}\right)$ and $+\star(k)f\left(\frac{n}{k}-1\right)$. To produce $-f(q)$, we need $\frac{n}{k} = q$ or $\frac{n}{k}-1 = q$:

\textbf{Case 1:} If $\frac{n}{k} = q$, then $k = \frac{n}{q} = p$. But $p$ is unique, so no other $k$ satisfies this.

\textbf{Case 2:} If $\frac{n}{k}-1 = q$, then $\frac{n}{k} = q+1$. This implies that $n$ has a proper divisor $\frac{n}{k}$ that is strictly greater than $q$. This is a contradiction, since we defined $q$ as the largest proper divisor of $n$.

It follows that the term $+f(q)$ cannot be canceled by any other term in the summation. Since $n$ is composite, $q = \frac{n}{p} > 1$. This means $f(q)$ is not fixed to $f(1)=1$, but can assume any value by choosing the function $f$. The polynomial for $\star(n)$ contains the uncanceled variable term $+f(q)$, so $\star(n)$ cannot identically equal the constant $-1$ for all valid functions $f$. By contraposition, if $\star(n) = -1$ identically for all valid functions $f$, then $n$ must be prime.

This proves the conjecture that $\star(p)=-1$ iff $p$ is prime.
\end{proof}


\begin{thebibliography}{9}
    \bibitem{chen2010} Chen, N. \textit{M\"{o}bius Inversion in Physics}. Tsinghua Report and Review in Physics -- Vol. 1. World Scientific Publishing (2010).
    \bibitem{sahoo} Sahoo, D. Ph.D. Dissertation, The Pennsylvania State University (2027).
    \bibitem{meissel1826} Meissel, E. “Observations quaedam in theoria numerorum.” \textit{Journal f\"{u}r die reine und angewandte Mathematik} (1826).
    \bibitem{flagg} Flagg, J. M. Kauffman, L. H. Sahoo, D. ``Laws Of Form and the Riemann Hypothesis." \textit{Series on Knots \& Everything}, Vol. 72 (2021).

\end{thebibliography}
\end{document}